\documentclass[10pt]{amsart}
\usepackage{amssymb,MnSymbol}
\usepackage{amsthm,amsmath}
\usepackage{epic,cite,color,tikz}

\title[Characterizations of biselective operations]{Characterizations of biselective operations}

\author{Jimmy Devillet}
\address{Mathematics Research Unit, University of Luxembourg, Maison du Nombre, 6, avenue de la Fonte, L-4364 Esch-sur-Alzette, Luxembourg}
\email{jimmy.devillet[at]uni.lu}

\author{Gergely Kiss}
\address{Mathematics Research Unit, University of Luxembourg, Maison du Nombre, 6, avenue de la Fonte, L-4364 Esch-sur-Alzette, Luxembourg}
\email{gergely.kiss[at]uni.lu}

\date{\today}

\theoremstyle{plain}
\newtheorem{theorem}{Theorem}[section]
\newtheorem{lemma}[theorem]{Lemma}
\newtheorem{proposition}[theorem]{Proposition}
\newtheorem{corollary}[theorem]{Corollary}
\newtheorem{fact}[theorem]{Fact}

\theoremstyle{definition}
\newtheorem{definition}[theorem]{Definition}

\theoremstyle{remark}
\newtheorem{remark}{Remark}

\newcommand{\R}{\mathcal{R}}

\newcommand{\ran}{\mathrm{ran}}
\newcommand{\id}{\mathrm{id}}

\begin{document}
\begin{abstract}
Let $X$ be a nonempty set and let $i,j \in \{1,2,3,4\}$. We say that
a binary operation $F:X^2\to X$ is $(i,j)$-selective if
$$
F(F(x_1,x_2),F(x_3,x_4))~=~F(x_i,x_j),
$$
for all $x_1,x_2,x_3,x_4\in X$. In this paper we provide
characterizations of the class of $(i,j)$-selective operations. We
also investigate some subclasses by adding algebraic properties such
as associativity or bisymmetry.
\end{abstract}



\maketitle

\section{Introduction}

Let $X$ be a nonempty set and let $i,j\in\{1,2,3,4\}$. We say that
an operation $F\colon X^2 \to X$ is \emph{$(i,j)$-selective} if
$$
F(F(x_1,x_2),F(x_3,x_4))~=~F(x_i,x_j),
$$
for all $x_1,x_2,x_3,x_4\in X$. Also, we say that an operation
$F\colon X^2 \to X$ is \emph{biselective} if there exist
$i,j\in\{1,2,3,4\}$ such that $F$ is $(i,j)$-selective. Among these
operations, those which are $(1,3)$-selective are of particular
interest as they are \emph{transitive}, that is, satisfy the
functional equation
$$
F(F(x,z),F(y,z)) ~=~ F(x,y),
$$
for all $x,y,z \in X$ (see, e.g., \cite{Acz2006,Kan2009} and the
references therein). Also, we easily see that $(1,4)$-selective
operations are \emph{bisymmetric}, that is, satisfy the functional
equation
$$
F(F(x,y),F(u,v)) ~=~ F(F(x,u),F(y,v)),
$$
for all $x,y,u,v \in X$ (see, e.g., \cite{Acz2006}).


In this paper we investigate the class of $(i,j)$-selective
operations for every $i,j\in \{1,2,3,4\}$. In particular, we
characterize these operations with and without additional properties
such as associativity or bisymmetry.

The paper is organized as follows. After presenting the main
definitions, we show some basic results about $(i,j)$-selective
operations in Section 2. In particular, we prove that
$(i,j)$-selective operations with $j<i$ are constant (see
Proposition \ref{prop:ji}) as well as (2,3)-selective operations
(see Proposotion \ref{prop:sym2}). We also show that characterizing
the $(i,j)$-selective operations is equivalent to characterizing the
$(5-j,5-i)$-selective operations (see Lemma \ref{lem:leri}). In
Section 3 we characterize the $(1,3)$-selective operations (see
Theorem \ref{thm:utr1}). In Section 4 we characterize the
$(1,4)$-selective operations (see Theorem \ref{thm:chdb}) and in
Section 5 we describe the $(1,2)$-selective operations in
conjunction with additional properties such as associativity.
Finally, in Section 6 we summarize the main results and present some
open questions and directions of further investigations.
\section{Preliminaries}\label{s2}

In this section we introduce some basic definitions and present some
preliminary results. 

\begin{definition}\label{de:def1}
An operation $F\colon X^2\to X$ is said to be
\begin{itemize}
\item \emph{idempotent} if $F(x,x)=x$ for all $x\in X$,
\item \emph{quasitrivial} (or \emph{selective}) if $F(x,y)\in\{x,y\}$ for all $x,y \in X$,
\item \emph{commutative} if $F(x,y) = F(y,x)$ for all $x,y \in X$,
\item \emph{anticommutative} if $\forall x,y\in X$: $F(x,y)=F(y,x) \Rightarrow x=y$,
\item \emph{associative} if
$$
F(x,F(y,z)) ~=~ F(F(x,y),z),
$$
for all $x,y,z \in X$,
\end{itemize}
\end{definition}

\begin{definition}
Let $F\colon X^2\to X$ be an operation.
\begin{itemize}
\item An element $e\in X$ is said to be a \emph{neutral element} of $F$ if
$F(e,x) = F(x,e) = x$ for all $x\in X$. It can be easily shown that
such a neutral element is unique.
\item An element $z\in X$ is said to be an \emph{annihilator} of $F$ if
$F(x,z) = F(z,x) = z$ for all $x\in X$. It can be easily shown that
such an annihilator is unique.
\item We denote the range of $F$ by $\ran(F)$. Clearly, $\ran(F)$ is nonempty since $X$ is nonempty.
\item An element $x\in X$ is said to be \emph{idempotent for $F$} if $F(x,x)=x$. We denote the set of all idempotent elements of $F$ by
$\id(F)$. Clearly, $\id(F)\subseteq \ran(F)$.
\end{itemize}
\end{definition}

Recall that a binary relation $R$ on $X$ is said to be
\begin{itemize}
\item \emph{reflexive} if $\forall x\in X$: $xRx$,
\item \emph{symmetric} if $\forall x,y\in X$: $xRy$ implies $yRx$,
\item \emph{transitive} if $\forall x,y,z\in X$: $xRy$ and $yRz$ implies $xRz$.
\end{itemize}

Recall also that an \emph{equivalence relation on $X$} is a binary
relation $\sim$ on $X$ that is reflexive, symmetric, and transitive.
For all $u\in X$, we use the notation $[u]_{\sim}$ to denote the
equivalence class of $u$, that is, $[u]_{\sim}=\{x\in X: x\sim u\}$.

Given $F \colon X^2 \to X$ we define the equivalence relation
$\sim_{F}$ on $X$ by
$$
x \sim_{F} y ~\Leftrightarrow~ F(x,x)~=~F(y,y) \qquad x,y\in X.
$$

\begin{fact}\label{fact:id}
If $F\colon X^2 \to X$ is an $(i,j)$-selective operation, then
$\id(F)\cap [x]_{\sim_F} = \{F(x,x)\}$ for all $x\in X$.
\end{fact}

\begin{proposition}\label{prop:ji}
An operation $F\colon X^2\to X$ is an $(i,j)$-selective operation
with $j<i$ if and only if $F$ is constant.
\end{proposition}

\begin{proof}
(Necessity) First, suppose that $F$ is $(4,1)$-selective (for
$(3,1)$-, $(4,2)$- and $(4,3)$-selective operations the proof is
similar).

By Fact \ref{fact:id}, we have $F(x,x)\in \id(F)$ for all $x\in X$.
If $x,y\in \id(F)$, then $$F(x,y)=F(F(x,x),F(y,y))=F(y,x),$$ by
$(4,1)$-selectiveness. Applying this, we get
$$x=F(x,x)=F(F(x,y),F(y,x))=F(F(y,x),F(x,y))=F(y,y)=y.$$
Thus, $|\id(F)|=1$ and we can assume that $\id(F) = \{u\}$. Hence, by Fact \ref{fact:id}, $F(x,x)=u$ for all $x\in X$. 
Using $(4,1)$-selectiveness, we get
$$F(x,y)=F(F(y,y),F(x,x))=F(u,u)=u, \quad x,y\in X.$$

\

Now we suppose that $F$ is $(2,1)$-selective (the case where $F$ is
$(4,3)$-selective can be dealt with similarly). By Fact
\ref{fact:id}, we have $F(x,x)\in \id(F)$ for all $x\in X$. If
$x,y\in \id(F)$, then
$$x=F(x,x)=F(F(x,x),F(y,y))=F(x,y)$$ by $(2,1)$-selectiveness.
Applying that $x=F(x,y)$ and $y=F(y,x)$ for all $x,y\in \id(F)$, we
obtain $$x=F(x,x)=F(F(x,y),F(x,y))=F(y,x)=y.$$ Thus, $|\id(F)|=1$
and we can assume that $\id(F) = \{u\}$. Hence, by Fact
\ref{fact:id}, $F(y,y)=u$ for all $y\in X$. Using also
$(2,1)$-selectiveness of $F$ we get
$$F(y,z)=F(F(z,y),F(z,y))=u, \quad y,z\in X.$$

(Sufficiency) Obvious.
\end{proof}

 In the following two propositions we deal with the case where $F$ is an $(i,i)$-selective operation with $i\in \{1,2,3,4\}$.

\begin{proposition}
An operation $F\colon X^2\to X$ is $(2,2)$-selective (resp. $(3,3)$-selective) if and only if $F|_{\R^2(F)}$ is constant and $F(x,x)\in \id(F)$ for all $x\in X$.
\end{proposition}
\begin{proof}

(Necessity) Suppose that $F$ is $(2,2)$-selective (the case where
$F$ is $(3,3)$-selective can be dealt with similarly). By Fact
\ref{fact:id}, $F(z,z)\in \id(F)$ for all $z\in X$.  If $x,y\in
\id(F)$, then
$$F(x,y)=F(F(x,x),F(y,y))=F(x,x)=x.$$
Using this and $(2,2)$-selectiveness we obtain
$$x=F(x,x)=F(F(x,y),F(x,y))=F(y,y)=y.$$
Thus, we can assume that $\id(F) = \{x\}$ and by Fact \ref{fact:id},
$F(y,y)=x$ for all $y\in X$. Now let us assume that $u,v\in
\ran(F)$. Then there exist $a,b,c,d\in X$ such that $u=F(a,b)$ and
$v=F(c,d)$. Using $(2,2)$-selectiveness we obtain
$$F(u,v)=F(F(a,b),F(c,d))=F(b,b)=x,$$ which proves the statement.

(Sufficiency) Obvious.
\end{proof}

\begin{proposition}\label{prop:11-selective}
An operation $F\colon X^2 \to X$ is $(1,1)$-selective (resp.
$(4,4)$-selective) if and only if the following conditions hold.
\begin{enumerate}
\item[(a)] $F(x,y) = F(x,x)$ (resp. $F(x,y)=F(y,y)$) for all $x,y\in \ran(F)$.
\item[(b)] $F(x,y) \in [x]_{\sim_{F}}$   (resp. $F(x,y)\in [y]_{\sim_{F}}$) for all $x,y\in X$.
\end{enumerate}
\end{proposition}

\begin{proof} Suppose that $F$ is $(1,1)$-selective (the case where $F$ is $(4,4)$-selective can be dealt with similarly).

(Necessity) If $x,y\in \ran(F)$, then there exists $a,b,c,d\in X$
such that $F(x,y)=F(F(a,b),F(c,d))=F(a,a)$ by $(1,1)$-selectiveness.
Also, by $(1,1)$-selectiveness,
$$F(x,x)=F(F(x,y), F(x,y))=F(F(a,a),F(a,a))=F(a,a),$$
which gives that $F(x,y)=F(x,x)$.

Also, by $(1,1)$-selectiveness, we have $F(F(x,y),F(x,y))=F(x,x)$,
which shows that $F(x,y)\in [x]_{\sim_F}$ for all $x,y\in X$.

(Sufficiency) Let $x,y,u,v\in X$. By condition $\textrm{(a)}$, we obtain
$$F(F(x,y),F(u,v)) = F(F(x,y),F(x,y)).$$
Also, by condition $\textrm{(b)}$, we obtain $F(F(x,y),F(x,y))=F(x,x)$.
\end{proof}

\begin{remark}
In Figure \ref{fig:1}, we illustrate the partitioning of $X$ by
$\sim_F$, where $F\colon X^2\to X$ is $(1,1)$-selective.
\end{remark}

\begin{figure}[!ht]
\begin{center}
\begin{tikzpicture}



\draw[fill=black] (0.5,0) circle (0.05); \draw[fill=black] (3.5,0)
circle (0.05);
\node at (0.5,3) {$\vdots$}; \draw[fill=black] (0.5,1) circle
(0.05); \draw[fill=black] (3.5,1) circle (0.05);
\draw[fill=black] (0.5,2) circle (0.05); \draw[fill=black] (3.5,2)
circle (0.05);
\node at (3.5,3) {$\vdots$};
 \draw[blue] (0.5,1.5) ellipse (0.7 and 2);
 \draw[red] (0.5,1) ellipse (0.6 and 1.5);
 \draw[blue] (3.5,1.5) ellipse (0.7 and 2);
 \draw[red] (3.5,0.5) ellipse (0.5 and 1);
 \node at (5,0) {$\dots$};
 \draw[black, thick] (-0.5,-0.5)--(-0.5, 0.5)--(4.5,0.5)--(4.5,-0.5)--cycle;

 \node at (-1,0) {$\id(F)$};
 \node at (0.5,0.25) {$F(x,x)$};
 \node at (3.5,0.25) {$F(y,y)$};
 \node[blue] at (0.7,4) {$[x]_{\sim_F}$};
 \node[red] at (-1,2) {$\ran(F)$};
 \node[blue] at (3.7,4) {$[y]_{\sim_F}$};
 \node[blue] at (5,4) {$\dots$};
 \node[blue] at (5,2) {$\dots$};
 \node[blue] at (5,1) {$\dots$};

\end{tikzpicture}
\caption{}\label{fig:1}
\end{center}
\end{figure}


\begin{definition}\label{def1}
We say that an operation $F\colon X^2\to X$ is \emph{dual
transitive} if
$$
F(F(x,y),F(x,z)) ~=~ F(y,z),
$$
for all $x,y,z \in X$.
\end{definition}

The following lemma shows a strong connection between $(i,j)$-selective and \\ $(5-j,5-i)$-selective operations. The proof is omitted as it is straightforward. 

\begin{lemma}\label{lem:leri} An operation $F:X^2\to X$ is $(i,j)$-selective (resp.\ transitive) if and only if the
operation $G:X^2\to X$ defined by $G(x,y)=F(y,x)$ for all $x,y\in X$
is $(5-j,5-i)$-selective (resp.\ dual transitive).
\end{lemma}

%

By Lemma \ref{lem:leri}, all the results for $(5-j,5-i)$-selective
operations can be deduced from those for $(i,j)$-selective
operations. Therefore we can focus only on $(1,2)$-, $(1,3)$-,
$(1,4)$-, and $(2,3)$-selective operations. Now we prove some useful
lemmas concerning these operations.

Recall that the {\it projection operations} are the binary
operations $\pi_1:X^2\to X$ and $\pi_2:X^2\to X$ defined by
$\pi_1(x,y)=x$ and $\pi_2(x,y)=y$ for all $x,y\in X$.

The following result provides a characterization of the
$(1,2)$-selective operations. Its proof is omitted as it is
straightforward.

\begin{lemma}\label{lem:s2}
An operation $F\colon X^2\to X$ is $(1,2)$-selective 
if and only if $F|_{\ran(F)^2}=\pi_1|_{\ran(F)^2}$. 
\end{lemma}
%

\begin{lemma}\label{lem:s1}
Let $F\colon X^2\to X$ be an operation that is $(1,4)$-selective
(resp.\ $(1,3)$-selective, $(2,3)$-selective). For all $x,y\in X$ we
have $F(x,y)=F(y,x)$ if and only if $F(x,x)=F(y,y)$. Moreover, if
there exist $x,y\in X$ such that any of the previous equalities
hold, then $F(x,y)=F(y,x)=F(x,x)=F(y,y)$.
\end{lemma}
\begin{proof}
We consider the case when $F$ is $(1,4)$-selective (the other cases
can be dealt with similarly).

(Necessity) Suppose that $F(x,y)=F(y,x)$, then
$$F(x,y)~=~F(F(x,y), F(x,y))~=~F(F(x,y), F(y,x))~=~F(x,x).$$
Similarly, we have that $F(x,y)=F(y,y)$.

(Sufficiency) Now assume that $F(x,x)=F(y,y)$, then
$$F(x,x)~=~F(F(x,x), F(x,x))~=~F(F(x,x), F(y,y))~=~F(x,y).$$
Similarly, we have that $F(x,x)=F(y,x)$.

The last statement of the lemma is now immediate.
\end{proof}

\begin{theorem}\label{prop:san}
Let $F\colon X^2\to X$ be an $(i,j)$-selective operation. Then the
following assertions are equivalent.
\begin{enumerate}
    \item[(i)] $F$ is commutative.
    \item[(ii)] $|\ran(F)|=1$ .
    \item[(iii)] $F$ has an annihilator.
    \item[(iv)] $|\id(F)|=1$.
\end{enumerate}
Moreover, we have that $F$ has a neutral element if and only if
$|X|=1$.
\end{theorem}

\begin{proof} By Lemma \ref{lem:leri} we only need to prove the result for $(i,j)$-selective operations where $(i,j)\in \{(1,2),(1,3),(1,4),(2,3)\}$. First, suppose that $F$ is $(1,4)$-selective (the cases where $F$ is $(1,3)$-selective or $(2,3)$-selective are similar).

$\textrm{(i)} \Rightarrow \textrm{(ii)}.$ If $F$ is commutative,
then by Lemma \ref{lem:s1}, $F(x,y)=F(x,x)$ for all $x,y\in X$. This
implies that $|\ran(F)|=1$.

$\textrm{(ii)} \Rightarrow \textrm{(iii)}.$ If $|\ran(F)|=1$, then
$F$ has clearly an annihilator.

$\textrm{(iii)} \Rightarrow \textrm{(iv)}.$ Let $a$ be the
annihilator of $F$. We clearly have that $a\in \id(F)$. Also, if
$x\in \id(F)$, then using $(1,4)$-selectiveness and the definition
of an annihilator, we get $ x = F(x,x) = F(F(x,a),F(a,x)) = F(a,a) =
a $ which shows that $|\id(F)|=1$.

$\textrm{(iv)} \Rightarrow \textrm{(i)}.$ If $|\id(F)|=1$, then
using Lemma \ref{lem:s1} we get that $F$ is commutative.

Let us now prove the last part of the statement. If $|X|=1$, then
$F$ has clearly a neutral element. Conversely, if $F$ has a neutral
element $e\in X$, then by Lemma \ref{lem:s1}, $x=F(x,e)=F(e,e)=e$
for all $x\in X$ which clearly implies that $|X|=1$.

\medskip

Now, suppose that $F$ is $(1,2)$-selective.

$\textrm{(i)} \Rightarrow \textrm{(ii)}.$ If $a,b\in R(F)$, then
using commutativity of $F$ and Lemma \ref{lem:s2} we get
$a=F(a,b)=F(b,a)=b$, which shows that $|\ran(F)|=1$.

$\textrm{(ii)} \Rightarrow \textrm{(iii)}.$ If $|\ran(F)|=1$, then
$F$ has clearly an annihilator.

$\textrm{(iii)} \Rightarrow \textrm{(iv)}.$ Let $a$ be the
annihilator of $F$. We clearly have that $a\in \id(F)$. Also, if
$x\in \id(F)$, then using $(1,2)$-selectiveness and the definition
of an annihilator, we get $ x = F(x,x) = F(F(x,x),F(a,a)) = F(x,a) =
a, $ which shows that $|\id(F)|=1$.

$\textrm{(iv)} \Rightarrow \textrm{(i)}.$ If $\id(F)=\{c\}$, then
using $(1,2)$-selectiveness we get
$$
F(x,y) ~=~ F(F(x,y),F(x,y)) ~=~ c ~=~ F(F(y,x),F(y,x)) ~=~ F(y,x),
$$
for all $x,y\in X$.

Let us now prove the last part of the statement. If $|X|=1$, then
$F$ has clearly a neutral element. Conversely, if $F$ has a neutral
element $e\in X$, then for all $x\in X$ we have $F(x,e)=x\in
\ran(F)$. On the other hand, by Lemma \ref{lem:s2}, $F(e,x)=e$ for
all $x\in \ran(F)$ which implies that $|X|=1$.
\end{proof}

As an important consequence of Theorem \ref{prop:san}, we provide
the characterization of $(2,3)$-selective operations.

\begin{proposition}\label{prop:sym2}
An operation $F\colon X^2 \to X$ is $(2,3)$-selective if and only if
$|\ran(F)|=1$.
\end{proposition}

\begin{proof}
(Necessity) We first show that $F(x,y)=F(y,x)$ for all $x,y\in X$.
Using four times $(2,3)$-selectiveness we obtain
\begin{eqnarray*}
  F(x,y) &=& F\left(F(x,x),F(y,y)\right) \\
  &=& F\left(F\left(F(y,x),F(x,y)\right),F\left(F(x,y),F(y,x)\right)\right) \\
  &=& F\left(F(x,y),F(x,y)\right) ~=~ F(y,x),
\end{eqnarray*}
which proves the commutativity of $F$. Now, it follows from Theorem
\ref{prop:san} that $|\ran(F)|=1$.

(Sufficiency) Obvious.
\end{proof}


\section{$(1,3)$-Selectiveness}

In the following result we provide a characterization of
$(1,3)$-selective operations. In the following $\sim_{F}$ denotes
the same equivalence relation as in Section \ref{s2}.

\begin{theorem}\label{thm:utr1}
Let $F:X^2\to X$ be an operation. Then, the following assertions are
equivalent.
\begin{enumerate}
\item[(i)] $F$ is $(1,3)$-selective.
\item[(ii)] $F(x,y)=F(u,v)\in [x]_{\sim_{F}}$ for all $x,y\in X$, $u\in [x]_{\sim_{F}}$, and $v\in [y]_{\sim_{F}}$.
\item[(iii)] $F(F(x,y),z) = F(x,z)$ and $F(x,F(y,z)) = F(x,y)$ for all $x,y,z\in X$.
\end{enumerate}
\end{theorem}

\begin{proof}
$\textrm{(i)} \Rightarrow \textrm{(ii)}.$ Let $x,y\in X$, $u\in
[x]_{\sim_{F}}$ and $v\in [y]_{\sim_{F}}$. Using
$(1,3)$-selectiveness we get
$$F(x,y)~=~F(F(x,x),F(y,y))~=~F(F(u,u),F(v,v))~=~F(u,v).$$
Also, using $(1,3)$-selectiveness, we get $F(F(x,y),F(x,y))=F(x,x),$
which shows that $F(x,y)=F(u,v)\in [x]_{\sim_{F}}$.

$\textrm{(ii)} \Rightarrow \textrm{(iii)}.$ Let $x,y,z\in X$.
Clearly, $x\in [x]_{\sim_{F}}$, $z\in [z]_{\sim_{F}}$, and by (ii)
we get $F(x,y)\in [x]_{\sim_{F}}$ and $F(y,z)\in [y]_{\sim_{F}}$.
Thus, by (ii), we obtain $F(F(x,y),z)=F(x,z)$ and
$F(x,F(y,z))=F(x,y)$.

$\textrm{(iii)} \Rightarrow \textrm{(i)}.$ Let $x,y,u,v\in X$. By
(iii), we obtain
$$F(F(x,y),F(u,v)) ~=~ F(F(x,y),u) ~=~ F(x,u),$$
which concludes the proof.
\end{proof}

\begin{remark}
In Figure \ref{fig:2}, we illustrate the partitioning of $X$ by
$\sim_F$, where $F\colon X^2\to X$ is $(1,3)$-selective.
\end{remark}

\begin{figure}[!ht]
\begin{center}
\begin{tikzpicture}



\draw[fill=black] (0.5,0) circle (0.05); \draw[fill=black] (2,0)
circle (0.05); \draw[fill=black] (3.5,0) circle (0.05);
\node at (0.5,3) {$\vdots$}; \draw[fill=black] (0.5,1) circle
(0.05); \draw[fill=black] (2,1) circle (0.05); \draw[fill=black]
(3.5,1) circle (0.05);
\node at (2,3) {$\vdots$}; \draw[fill=black] (0.5,2) circle (0.05);
\draw[fill=black] (2,2) circle (0.05); \draw[fill=black] (3.5,2)
circle (0.05);
\node at (3.5,3) {$\vdots$};
 \draw[blue] (0.5,1.5) ellipse (0.7 and 2);
 \draw[blue] (2,1.5) ellipse (0.7 and 2);
 \draw[blue] (3.5,1.5) ellipse (0.7 and 2);
 \node at (5,0) {$\dots$};
 \draw[black, thick] (-0.5,-0.5)--(-0.5, 0.5)--(4.5,0.5)--(4.5,-0.5)--cycle;

 \node at (-1,0) {$\id(F)$};
 \node at (0.5,0.25) {$F(x,x)$};
 \node at (2,0.25) {$F(y,y)$};
 \node at (3.5,0.25) {$F(z,z)$};
 \node[blue] at (0.7,4) {$[x]_{\sim_F}$};
 \node[blue] at (2.2,4) {$[y]_{\sim_F}$};
 \node[blue] at (3.7,4) {$[z]_{\sim_F}$};
 \node[blue] at (5,4) {$\dots$};
 \node[blue] at (5,2) {$\dots$};
 \node[blue] at (5,1) {$\dots$};


\end{tikzpicture}
\caption{}\label{fig:2}
\end{center}
\end{figure}

\begin{corollary}\label{cor:utr1}
If $F\colon X^2\to X$ is a $(1,3)$-selective operation, then
$F(x,x)=F(y,z)$ for all $x\in X$ and $y,z\in [x]_{\sim_{F}}$.
\end{corollary}

Now if we assume that $F$ is surjective, then we easily derive the
following characterization from Theorem \ref{thm:utr1}.

\begin{corollary}\label{thm:utr2}
An operation $F\colon X^2 \to X$ is $(1,3)$-selective and surjective
if and only if the following conditions hold.
\begin{enumerate}
    \item[(a)] $F(x,y)=F(u,v)\in [x]_{\sim_{F}},$ for all $x,y\in X$, $u\in [x]_{\sim_{F}}$, and $v\in [y]_{\sim_{F}}$.
    \item[(b)] For every $x\in X$ there exists $a,b\in X$ such that $F(F(a,a),F(b,b))=F(a,b)=x\in [a]_{\sim_{F}}$.
\end{enumerate}
\end{corollary}

\begin{remark}
Using condition \textrm{(a)} of Corollary \ref{thm:utr2}, we can reformulate
condition \textrm{(b)} of Corollary \ref{thm:utr2} in the following way. If
a $(1,3)$-selective operation $F$ is surjective, then for every
$x\in X$ there exists $y\in X$ such that $F(x,y)=x$.
\end{remark}

The following result shows that associativity and bisymmetry are
equivalent under $(1,3)$-selectiveness.
\begin{proposition}\label{prop:asbi}
Let $F:X^2 \to X$ be an $(1,3)$-selective operation. Then the
following assertions are equivalent.
\begin{enumerate}
    \item[(i)] $F$ is associative,
    \item[(ii)] $F(x,y)=F(x,z)$ for all $x,y,z\in X$,
    \item[(iii)] $F$ is bisymmetric,
    \item[(iv)] $F|_{\ran(F)\times X}=\pi_1|_{\ran(F)\times X}$.
\end{enumerate}
 \end{proposition}
\begin{proof}
$\textrm{(i)}\Rightarrow \textrm{(ii)}$. This follows from Theorem
\ref{thm:utr1}.

$\textrm{(ii)}\Rightarrow \textrm{(iii)}$. Let $x,y,u,v\in X$. Using
$(1,3)$-selectiveness and condition $(ii)$, we obtain
$$F(F(x,y),F(u,v))~=~F(x,u)~=~F(x,y)~=~F(F(x,u),F(y,v)),$$
which shows that $F$ is bisymmetric.

$\textrm{(iii)}\Rightarrow \textrm{(iv)}$. Let $x,y,z\in X$. By
$(1,3)$-selectiveness, $[z]_{\sim_{F}}=[F(z,y)]_{\sim_{F}}$ for all
$y\in X$. Using Theorem \ref{thm:utr1}, bisymmetry and
(1,3)-selectiveness, we obtain
$$F(F(x,y),z)~=~F(F(x,y),F(z,y))~=~F(F(x,z),F(y,y))=F(x,y).$$

$\textrm{(iv)}\Rightarrow \textrm{(i)}$. Let $x,y,z\in X$. By
Theorem \ref{thm:utr1} we have $F(x,y), F(x,z)\in [x]_{\sim_{F}}$.
Thus, using Theorem \ref{thm:utr1} and condition $(iv)$, we obtain
\begin{eqnarray*}
F(F(x,y),z) ~=~ F(x,z) ~=~ F(F(x,z),F(y,z)) ~=~ F(x,F(y,z)).
\end{eqnarray*}
\end{proof}

The following result provides characterizations of idempotent
$(1,3)$-selective operations.
\begin{proposition}
Let $F\colon X^2 \to X$ be an $(1,3)$-selective operation. Then the
following assertions are equivalent:
\begin{enumerate}
    \item[(i)] $F$ is quasitrivial.
    \item[(ii)] $F$ is idempotent.
    \item[(iii)] $|[x]_{\sim_{F}}|=1$ for all $x\in X$.
     \item[(iv)] $F=\pi_1$.
    \item[(v)] $F$ is anticommutative.
\end{enumerate}
\end{proposition}
\begin{proof}

$\textrm{(i)}\Rightarrow \textrm{(ii)}$. Obvious.

$\textrm{(ii)}\Rightarrow \textrm{(iii)}$. Obvious.

$\textrm{(iii)}\Rightarrow \textrm{(iv)}$. Let $x,y\in X$ with
$x\neq y$. Since $|[x]_{\sim_{F}}|=|[y]_{\sim_{F}}|=1$, we clearly
have that $x$ (resp. $y$) is the unique element of $[x]_{\sim_{F}}$
(resp. $[y]_{\sim_{F}}$). Also, by Theorem \ref{thm:utr1} we have
$F(x,y)\in [x]_{\sim_{F}}$ and $F(y,x)\in [y]_{\sim_{F}}$ and hence
$F(x,y)=x$ and $F(y,x)=y$.

$\textrm{(iv)}\Rightarrow \textrm{(v)}$. Obvious.

$\textrm{(v)}\Rightarrow \textrm{(i)}$. We proceed by contradiction.
Suppose that there exist $x,y\in X$ such that $F(x,y)\notin
\{x,y\}$. We clearly have $x\in [x]_{\sim_{F}}$ and by Theorem
\ref{thm:utr1} we have $F(x,y)\in [x]_{\sim_{F}}$. Thus, using
Corollary \ref{cor:utr1}, we obtain
$F(F(x,y),x)=F(x,x)=F(x,F(x,y))$, a contradiction.
\end{proof}

As an application of the structural description (see Theorem
\ref{thm:utr1} and Figure \ref{fig:2}) we can get the following
results.
\begin{proposition}\label{prop:fi1}
Let $F$ be a $(1,3)$-selective operation on a finite $X$. Then
$$|\ran(F)|\le |\id(F)|^2.$$
\end{proposition}

\begin{proof}
By Fact \ref{fact:id}, since $F$ is $(1,3)$-selective,  $F(y,y),
F(z,z)\in \id(F)$ for all $y,z\in X$. This clearly implies that the
number of ordered pairs of $\id(F)$ cannot be smaller than
$|\ran(F)|$.
\end{proof}

The set $\id(F)$ can be listed as $$\id(F)=\{x_1, \dots, x_k\}$$ for
some $k\ge 1$. We denote the cardinality of $[x_i]_{\sim_{F}}$
by $l_i$ for all $i \in \{1, \dots, k\}$.

\begin{corollary}\label{cornb}
Let $F$ be a surjective $(1,3)$-selective operation on a finite $X$
of cardinality $n\ge 1$ and $k=|\id(F)|$. Then $n\le k^2$ and $l_i\le k$
for all $1\le i\le k$.
\end{corollary}

\begin{proof}
This follows from Proposition \ref{prop:fi1} and Corollary
\ref{thm:utr2}.
\end{proof}

Let $s_k(n)$ denote the number of $(1,3)$-selective and surjective
operations on an $n$-element set $X$ with $|\id(F)|=k$. By Corollary \ref{cornb}, we have
$s_k(n)=0$ if $n> k^2$. Finding a closed-form expression for the number of
$(1,3)$-selective or surjective $(1,3)$-selective operations seems
hopeless. As an illustration of the characterization given in
Theorem \ref{thm:utr1}, we calculate $s_k(k^2)$.

\begin{proposition}
For all integer $k\ge 1$, we have
$$s_k(k^2)=\frac{k^2!}{((k-1)!)^{k-1}}.$$
\end{proposition}
\begin{proof}
We have $n=k^2$ and $k=|\id(F)|$. Hence by Corollary
\ref{cornb}, $l_i=k$ for all $1\le i\le k$ (see Figure \ref{fig:4}).

\begin{figure}[!ht]
\begin{center}
\begin{tikzpicture}

\draw[fill=black] (0,0) circle (0.05); \draw[fill=black] (1,0)
circle (0.05);
\draw[fill=black] (3,0) circle (0.05);
\node at (0,2) {$\vdots$}; \draw[fill=black] (0,1) circle (0.05);
\draw[fill=black] (1,1) circle (0.05);
\draw[fill=black] (3,1) circle (0.05);
\node at (1,2) {$\vdots$};
\draw[fill=black] (0,3) circle (0.05); \draw[fill=black] (1,3)
circle (0.05);
\draw[fill=black] (3,3) circle (0.05);
\node at (3,2) {$\vdots$};

 \draw[blue] (0,1.5) ellipse (0.3 and 2);
 \draw[blue] (1,1.5) ellipse (0.3 and 2);
 \draw[blue] (3,1.5) ellipse (0.3 and 2);
 \node at (2,0) {$\dots$};
 \draw[black, thick] (-0.5,-0.5)--(-0.5, 0.5)--(3.5,0.5)--(3.5,-0.5)--cycle;

 \node at (-1,0) {$\id(F)$};
 \node at (0,0.25) {$x_1$};
 \node at (1,0.25) {$x_2$};
 \node at (3,0.25) {$x_k$};
 \node[blue] at (0.2,4) {$[x_1]_{\sim_F}$};
 \node[blue] at (1.2,4) {$[x_2]_{\sim_F}$};
 \node[blue] at (3.2,4) {$[x_k]_{\sim_F}$};
 \node[blue] at (2,1) {$\dots$};
 \node[blue] at (2,2) {$\dots$};
 \node[blue] at (2,3) {$\dots$};
 \node[blue] at (2,4) {$\dots$};

 \node at (4,0) {$1$};
 \node at (4,1) {$2$};
 \node at (4,2) {$\vdots$};
 \node at (4,3) {$k$};
\end{tikzpicture}

\caption{} \label{fig:4}
\end{center}
\end{figure}

Let $X$ be a set with $k^2$ elements. First we choose the
equivalence classes for $\sim_F$. All of them have $k$
elements, thus this can be made by the multinomial coefficient
$\frac{k^2!}{(k!)^k}$. Now we choose one element from each class
that is also in $\id(F)$. This can be done in $k^k$ different ways.

By Theorem \ref{thm:utr1} and surjectivity of $F$, for all $x_i\in
\id(F)$ the elements of $[x_i]_{\sim_F}$ are of the form in $F(x_i,
x_j)$ for some $x_j\in \id(F)$. Since $l_i=|\id(F)|$, this implies
that $F(x_i, \cdot )$ is a bijection between $\id(F)$ and
$[x_i]_{\sim_F}$ with a fixed point $x_i\in \id(F)$. Thus for each
$i\in \{1, \dots, k\}$ there are $(k-1)!$ different permutations.
Consequently,
$$s_k(k^2)=\frac{k^2!}{(k!)^k}\cdot k^k\cdot k(k-1)!=
\frac{k^2!\cdot k}{((k-1)!)^{k-1}}.$$
\end{proof}
\begin{remark}
We observe that the number of isomorphism type of $(1,3)$-selective
and surjective operations on a $k^2$-element set $X$ with
$|\id(F)|=k$ is $1$.
\end{remark}

\section{$(1,4)$-Selectiveness}

In this section we characterize the class of operations $F\colon X^2
\to X$ that are $(1,4)$-selective (see Theorem \ref{thm:chdb}).

Clearly, any operation $F\colon X^2 \to X$ that is diagonal
bisymmetric is bisymmetric. The following result provides a
characterization and partial description of the class of
$(1,4)$-selective operations.

\begin{proposition}\label{lem:ass}
Let $F\colon X^2\to X$ be an operation. Then, the following
assertions are equivalent.
\begin{enumerate}
\item[(i)] $F$ is $(1,4)$-selective.
\item[(ii)] $F|_{R(F)^2}$ is $(1,4)$-selective and $F(x,y)=F(F(x,x),F(y,y))$ for all $x,y\in X$.
\item[(iii)] $F(F(x,y),z) ~=~ F(x,F(y,z)) ~=~ F(x,z)$ for all $x,y,z\in X$.
\end{enumerate}
\end{proposition}

\begin{proof}
$\textrm{(i)} \Rightarrow \textrm{(ii)}$. Obvious.

$\textrm{(ii)} \Rightarrow \textrm{(iii)}$. Let $x,y,z \in X$ and
let us prove that $F(F(x,y),z)=F(x,z)$ (the other case can be dealt
with similarly). Using our assumptions, we get
\begin{eqnarray*}
F(F(x,y),z) &=& F(F(F(x,y),F(x,y)),F(z,z)) \\
&=& F(F(F(x,y),F(x,y)),F(F(z,z),F(z,z))) \\
&=& F(F(x,y),F(z,z)) \\
&=& F(F(F(x,x),F(y,y)),F(F(z,z),F(z,z))) \\
&=& F(F(x,x),F(z,z)) ~=~ F(x,z),
\end{eqnarray*}
which concludes the proof.

$\textrm{(iii)} \Rightarrow \textrm{(i)}$. Let $x,y,u,v \in X$. By
assumption, we have $F(F(x,y),F(u,v)) = F(x,F(u,v)) = F(x,v),$ which
shows that $F$ is $(1,4)$-selective.
\end{proof}

\begin{corollary}\label{cor:ass}
If $F\colon X^2\to X$ is $(1,4)$-selective, then it is associative.
\end{corollary}




As another consequence we can prove the following.
\begin{proposition}\label{thm:main}
An operation $F\colon X^2 \to X$ is $(1,4)$-selective and
quasitrivial if and only if $F = \pi_1$ or $F = \pi_2$.
\end{proposition}

\begin{proof}
(Necessity) Let $x,y \in X, x \neq y$, then we have $F(x,y) \in
\{x,y\}$. We can suppose without loss of generality that $F(x,y) =
x$ (the other case can be dealt with similarly). Then, by
Proposition \ref{lem:ass} we have $F(F(x,z),y)=F(x,y)=x$ for all
$z\in X$. Thus, by quasitriviality we have $F(x,z)=x$ for all $z\in
X$.

(Sufficiency) Obvious.
\end{proof}

\begin{remark}\label{rem}
We observe that quasitriviality cannot be relaxed into idempotency
in Theorem \ref{thm:main}. Indeed, let us consider $X = \{a,b,c,d\}$
and the operation $F \colon X^2 \to X$ defined by $F(a,u)=F(c,u)=a$
and $F(b,u)=F(d,u)=d$ for all $u\in\{a,d\}$ and by $F(a,v)=F(c,v)=c$
and $F(b,v)=F(d,v)=b$ for all $v \in \{b,c\}$. It is not difficult
to see that $F$ is idempotent and $(1,4)$-selective, however it is
neither
$\pi_1$ nor $\pi_2$. It is also not hard to show that any idempotent and $(1,4)$-selective operation on $X$, that is not quasitrivial, can be given as the previous operation after a suitable permutation of the elements $a,b,c,d$. 
\end{remark}



Now we provide a characterization of $(1,4)$-selective operations
using equivalence relations.

Given $F \colon X^2 \to X$ we define two binary relations
$\sim_{F,1}$ and $\sim_{F,2}$ on $X$ by
$$
x \sim_{F,1} y ~\Leftrightarrow~ F(x,y)=x \qquad x,y\in X,
$$
and
$$
x \sim_{F,2} y ~\Leftrightarrow~ F(x,y)=y \qquad x,y\in X.
$$

Thus, given $F\colon X^2 \to X$ we have that
$F|_{\{x,y\}^2}=\pi_1|_{\{x,y\}^2}$ (resp.\
$F|_{\{x,y\}^2}=\pi_2|_{\{x,y\}^2}$) for all $x,y\in X$ with $x
\sim_{F,1} y$ (resp.\ $x \sim_{F,2} y$).

\begin{remark}
We observe that for all $F\colon X^2 \to X$ the binary relation
$\sim_{F,2}$ on $X$ was already introduced in \cite[Definition
2.4]{Can2012} and was called the \emph{trace} of $F$.
\end{remark}

\begin{lemma}\label{lem:sim}
Let $F\colon X^2\to X$ be a $(1,4)$-selective operation. Then, the
following assertions hold.
\begin{enumerate}
\item[(a)] $\ran(F)=\id(F)$.
\item[(b)] $\sim_{F,1}$ and $\sim_{F,2}$ are transitive binary relations on
$X$.
\end{enumerate}
\end{lemma}

\begin{proof}
$\textrm{(a)}$. If $x\in \id(F)$, then we clearly have that $x\in
\ran(F)$. Conversely, if $x\in \ran(F)$, then there exist $a,b\in X$
such that $x=F(a,b)$. Thus, using $(1,4)$-selectiveness, we get
$F(x,x) = F(F(a,b),F(a,b)) = F(a,b) = x,$ which concludes the proof.

$\textrm{(b)}$. We only prove that $\sim_{F,1}$ is transitive (the
other case can be dealt with similarly). Let $x,y,z\in X$ such that
$x\sim_{F,1}y$ and $y\sim_{F,1}z$, that is, $F(x,y)=x$ and
$F(y,z)=y$. Using diagonal bisymmetry, we get
$$F(x,z)~=~F(F(x,y),F(y,z))~=~F(x,y)~=~x,$$
that is, $x\sim_{F,1}z$.
\end{proof}

\begin{remark}
We observe that Lemma \ref{lem:sim}(ii) is still valid if the
operation $F$ is only associative.
\end{remark}

%

\begin{proposition}\label{prop:id}
Let $F\colon X^2\to X$ be $(1,4)$-selective. Then, the following
assertions are equivalent.
\begin{enumerate}
\item[(i)] $F$ is anticommutative.
\item[(ii)] $F$ is surjective.
\item[(iii)] $F$ is idempotent.
\item[(iv)] $\sim_{F,1}$ is an equivalence relation on X.
\item[(v)] $\sim_{F,2}$ is an equivalence relation on X.
\end{enumerate}
\end{proposition}

\begin{proof}
$\textrm{(i)} \Rightarrow \textrm{(ii)}.$ Suppose to the contrary
that $F$ is not surjective and hence not idempotent. Thus, there
exists $x\in X$ such that $F(x,x)\neq x$. Using Proposition
\ref{lem:ass}, we obtain $F(F(x,x),x) = F(x,x) = F(x,F(x,x)),$ a
contradiction.

$\textrm{(ii)} \Rightarrow \textrm{(iii)}.$ This follows from Lemma
\ref{lem:sim}\textrm{(a)}.

$\textrm{(iii)} \Rightarrow \textrm{(iv)}.$ $\sim_{F,1}$ is clearly
reflexive since $F$ is idempotent. Also, by Lemma \ref{lem:sim}\textrm{(b)},
we have that $\sim_{F,1}$ is transitive. Now, let us show that
$\sim_{F,1}$ is symmetric. Let $x,y \in X$ such that $x \sim_{F,1}
y$, that is, $F(x,y) = x$. Then, using idempotency and
$(1,4)$-selectiveness, we get
$$F(y,x) ~=~ F(F(y,y),F(x,y)) ~=~ F(y,y) ~=~ y,$$
that is, $y \sim_{F,1} x$.

$\textrm{(iv)} \Rightarrow \textrm{(v)}.$ $\sim_{F,2}$ is clearly
reflexive since $\sim_{F,1}$ is reflexive. Also, by Lemma
\ref{lem:sim}\textrm{(b)} we have that $\sim_{F,2}$ is transitive. Now, let
us show that $\sim_{F,2}$ is symmetric. Let $x,y\in X$ such that
$x\sim_{F,2}y$, that is, $F(x,y)=y$. By Proposition \ref{lem:ass},
we have that $y\sim_{F,2}F(y,x)$ and by transitivity of $\sim_{F,2}$
we get $x\sim_{F,2}F(y,x)$, that is, $F(x,F(y,x))=F(y,x)$. By
Proposition \ref{lem:ass} and symmetry of $\sim_{F,1}$, we also have
that $x\sim_{F,1}F(y,x)$, that is, $F(x,F(y,x))=x$. Hence
$F(y,x)=x$, that is, $y\sim_{F,2}x$.

$\textrm{(v)} \Rightarrow \textrm{(i)}.$ Let $x,y\in X$ such that
$F(x,y) = F(y,x)$. Since $\sim_{F,2}$ is an equivalence relation on
$X$, it follows that $F$ is idempotent. Thus, using
$(1,4)$-selectiveness, we obtain
$$
x~=~F(x,x)~=~F(F(x,y),F(y,x))~=~F(F(y,x),F(x,y))~=~F(y,y)~=~y,
$$
which concludes the proof.
\end{proof}

The following result can be easily derived from Lemma
\ref{lem:sim}\textrm{(a)} and Proposition \ref{prop:id}.

\begin{corollary}\label{cor:main}
If $F\colon X^2 \to X$ is $(1,4)$-selective, then $F|_{\ran(F)^2}$
satisfies any of the conditions $(i)-(v)$ of Proposition
\ref{prop:id}.
\end{corollary}

The following lemma is a consequence of \cite[Lemma 1]{McL54} and
\cite[Lemma 2]{Kim58}.

\begin{lemma}\label{lem:kim}
Let $F\colon X^2 \to X$ be associative and idempotent. Then, the
following assertions are equivalent.
\begin{enumerate}
\item[(i)] $F(F(x,y),z) = F(x,z)$ for all $x,y,z\in X$.
\item[(ii)] $F(F(x,y),x) = x$ for all $x,y\in X$.
\item[(iii)] $F$ is anticommutative.
\end{enumerate}
\end{lemma}

\begin{proposition}
An operation $F\colon X^2 \to X$ is $(1,4)$-selective and satisfies
any of the conditions $(i)-(v)$ of Proposition \ref{prop:id} if and
only if it is associative, idempotent, and satisfies any of the
conditions $(i)-(iii)$ of Lemma \ref{lem:kim}.
\end{proposition}

\begin{proof}
(Necessity) This follows from Corollary \ref{cor:ass} and
Proposition \ref{prop:id}.

(Sufficiency) This follows from Lemma \ref{lem:kim} and Propositions
\ref{lem:ass} and \ref{prop:id}.
\end{proof}


%

\begin{proposition}\label{prop:cdb}
Let $F\colon X^2 \to X$ be an operation. The following assertions are equivalent.
\begin{enumerate}
\item[(i)] $F$ is $(1,4)$-selective and satisfies
any of the conditions $(i)-(v)$ of Proposition \ref{prop:id}.
\item[(ii)] $\sim_{F,1}$ and $\sim_{F,2}$ are equivalence relations on
$X$ such that $[x]_{\sim_{F,2}}\cap [y]_{\sim_{F,1}} = \{F(x,y)\}$
for all $x,y\in X$.
\item[(iii)] The following conditions hold.
\begin{enumerate}
\item[(a)] $\sim_{F,1}$ and $\sim_{F,2}$ are equivalence relations on $X$.
\item[(b)] For all $x,y,z\in X$ such that $x\in [y]_{\sim_{F,1}}$ there exists a unique $u\in [z]_{\sim_{F,1}}$ such that $x\sim_{F,2} u$.
\item[(c)] $F(x,y) = F(x,z)$ for all $x,y,z\in X$ such that $y\sim_{F,1} z$.
\end{enumerate}
\end{enumerate}
\end{proposition}

\begin{proof}
$\textrm{(i)}\Rightarrow \textrm{(ii)}$: By Proposition \ref{prop:id} we have that $\sim_{F,1}$
and $\sim_{F,2}$ are equivalence relations on $X$. Let $x,y\in X$
and let us show that $[x]_{\sim_{F,2}}\cap [y]_{\sim_{F,1}} =
\{F(x,y)\}$. By Proposition \ref{lem:ass} we have $F(x,y)\in
[x]_{\sim_{F,2}}\cap [y]_{\sim_{F,1}}$. Also, if $z \in
[x]_{\sim_{F,2}}\cap [y]_{\sim_{F,1}}$, then $z \sim_{F,1}\cap
\sim_{F,2} F(x,y)$ which by definition implies that $z=F(x,y)$.

$\textrm{(ii)}\Rightarrow \textrm{(iii)}$: Condition $\textrm{(a)}$ is clearly satisfied.

Let $x,y,z\in X$ such that $x\in [y]_{\sim_{F,1}}$. By $\textrm{(ii)}$, we have $[x]_{\sim_{F,2}}\cap [z]_{\sim_{F,1}} = \{F(x,z)\}$, which proves condition $\textrm{(b)}$.

Let $x,y,z\in X$ such that $y\sim_{F,1} z$, that is, $[y]_{\sim_{F,1}} = [z]_{\sim_{F,1}}$. By $\textrm{(ii)}$ and the previous assumption, we get
$$
\{F(x,y)\} = [x]_{\sim_{F,2}}\cap [y]_{\sim_{F,1}} = [x]_{\sim_{F,2}}\cap [z]_{\sim_{F,1}} = \{F(x,z)\},
$$
which proves condition $\textrm{(c)}$.

$\textrm{(iii)}\Rightarrow \textrm{(i)}$: Since $\sim_{F,1}$ and
$\sim_{F,2}$ are equivalence relations on $X$, it follows that $F$
is idempotent. Let $x,y,u,v\in X$ and let us show that
$F(F(x,y),F(u,v))=F(x,v)$. We clearly have that $t\in
[t]_{\sim_{F,1}}$ for all $t\in X$. By conditions $\textrm{(b)}$ and
$\textrm{(c)}$, we have that there exists a unique $s\in
[y]_{\sim_{F,1}}$ such that $F(x,y)=F(x,s)=s$, that is, $x\sim_{F,2}
s$. Also, by conditions $\textrm{(b)}$ and $\textrm{(c)}$, we have
that there exists a unique $t\in [v]_{\sim_{F,1}}$ such that
$F(u,v)=F(u,t)=t$, that is, $u\sim_{F,2} t$. Moreover, by conditions
$\textrm{(b)}$ and $\textrm{(c)}$, we have that there exists a
unique $z\in [v]_{\sim_{F,1}}$ such that $F(s,t)=F(s,z)=z$, that is,
$s\sim_{F,2} z$. By transitivity of $\sim_{F,2}$ we have that
$x\sim_{F,2} z$ and by condition $\textrm{(b)}$ we have that $z$ is
unique. Thus, we obtain $F(F(x,y),F(u,v))=F(x,v)=F(x,z)=z$ which
concludes the proof.
\end{proof}

\begin{remark}
In Proposition \ref{prop:cdb}\textrm{(iii)}, conditions $\textrm{(b)}$ and
$\textrm{(c)}$ can be replaced by the following two conditions.
\begin{enumerate}
\item[(b')] For all $x,y,z\in X$ such that $x\in [y]_{\sim_{F,2}}$ there exists a unique $u\in [z]_{\sim_{F,2}}$ such that $x\sim_{F,1} u$.
\item[(c')] $F(y,x) = F(z,x)$ for all $x,y,z\in X$ such that $y\sim_{F,2} z$.
\end{enumerate}
\end{remark}

The following corollary is an equivalent form of Proposition
\ref{prop:cdb}.

\begin{corollary}\label{cor:chiddb}
An operation $F\colon X^2 \to X$ is $(1,4)$-selective and satisfies
any of the conditions $(i)-(v)$ of Proposition \ref{prop:id} if and
only if the following conditions hold.
\begin{enumerate}
\item[(i)] $\sim_{F,1}$ and $\sim_{F,2}$ are equivalence relations on $X$ and for all $x,y\in X$ there exists a bijection $f\colon [x]_{\sim_{F,1}} \to [y]_{\sim_{F,1}}$ defined by
    $$f(u)~=~v ~\Leftrightarrow~ u~\sim_{F,2}~v, \qquad u\in [x]_{\sim_{F,1}}, v\in [y]_{\sim_{F,1}}.$$
\item[(ii)] $F(x,y) = F(x,z)$ for all $x,y,z\in X$ such that $y\sim_{F,1} z$.
\end{enumerate}
\end{corollary}

According to Corollary \ref{cor:chiddb}, any $(1,4)$-selective and
idempotent operation $F\colon X^2 \to X$ can be represented in a
grid form as follows. Two elements $x,y\in X$ belong to the same
column (resp.\ row) if and only if $x\sim_{F,1} y$ (resp.\
$x\sim_{F,2} y$). Conversely, any operation $F\colon X^2 \to X$ such
that $\sim_{F,1}$ and $\sim_{F,2}$ are equivalence relations and
that can be represented in such a grid form with the convention that
$F(x,y) = F(x,z)$ for all $x,y,z\in X$ such that $y\sim_{F,1} z$, is
$(1,4)$-selective and idempotent (see Figure \ref{fig:3}).

\begin{figure}[!ht]
\begin{center}
\begin{tikzpicture}



\draw[fill=black] (0,0) circle (0.05); \draw[fill=black] (1,0)
circle (0.05); \draw[fill=black] (2,0) circle (0.05);
\draw[fill=black] (3,0) circle (0.05);
\node at (0,4) {$\vdots$}; \draw[fill=black] (0,1) circle (0.05);
\draw[fill=black] (1,1) circle (0.05); \draw[fill=black] (2,1)
circle (0.05); \draw[fill=black] (3,1) circle (0.05);
\node at (1,4) {$\vdots$}; \draw[fill=black] (0,2) circle (0.05);
\draw[fill=black] (1,2) circle (0.05); \draw[fill=black] (2,2)
circle (0.05); \draw[fill=black] (3,2) circle (0.05);
\node at (2,4) {$\vdots$}; \draw[fill=black] (0,3) circle (0.05);
\draw[fill=black] (1,3) circle (0.05); \draw[fill=black] (2,3)
circle (0.05); \draw[fill=black] (3,3) circle (0.05);
\node at (3,4) {$\vdots$};
 \draw[black] (0,2) ellipse (0.3 and 2.5);
 \draw[black] (1,2) ellipse (0.3 and 2.5);
 \draw (2,2) ellipse (0.3 and 2.5);
 \draw (3,2) ellipse (0.3 and 2.5);
 \node at (4,0) {$\dots$};
 \node at (4,1) {$\dots$};
 \node at (4,2) {$\dots$};
 \node at (4,3) {$\dots$};
 \node at (4,4) {$\udots$};
 \draw[black] (2,0) ellipse (2.5 and 0.3);
 \draw[black] (2,1) ellipse (2.5 and 0.3);
 \draw[black] (2,2) ellipse (2.5 and 0.3);
 \draw[black] (2,3) ellipse (2.5 and 0.3);

 \node at (5.3,0) {$[y]_{\sim_{F,2}}$};
 \node at (5.3,3) {$[x]_{\sim_{F,2}}$};
\draw[->] (-0.9,-0.9)--(-0.05,-0.05); \node at (-1.05,-1.05)
{$F(y,x)$}; \draw[->] (-0.85,3)--(-0.07,3);
 \node at (-1.05,3) {$x$};
 \draw[->] (2,-0.85)--(2,-0.08);
 \node at (2,-1.05) {$y$};
 \draw[->] (-0.85,3)--(-0.08,3);
 \node at (-1.05,3) {$x$};
 \draw[->] (3.7,4.7)--(2.1,3.1);
\node at (4,5) {$F(x,y)$};

 \node at (0.2,5) {$[x]_{\sim_{F,1
 }}$};
 \node at (2.2,5) {$[y]_{\sim_{F,1}}$};


\end{tikzpicture}
\caption{}\label{fig:3}
\end{center}
\end{figure}

The following result, which is an immediate consequence of
Propositions \ref{lem:ass}, \ref{prop:id}, \ref{prop:cdb} and
Corollaries \ref{cor:main} and \ref{cor:chiddb}, provides a characterization of
$(1,4)$-selective operations.


\newpage

\begin{theorem}\label{thm:chdb}
Let $F\colon X^2 \to X$ be an operation and let $F'=F|_{\ran(F)^2}$.
Then, the following assertions are equivalent.
\begin{enumerate}
\item[(i)] $F$ is $(1,4)$-selective
\item[(ii)] $F'$ satisfies the conditions $(i)-(iii)$ of Proposition \ref{prop:cdb} and $F(u,v)=F(F(u,u),F(v,v))$ for all $u,v\in X$.
\item[(iii)] $F'$ satisfies the conditions $(i)$ and $(ii)$ of Corollary \ref{cor:chiddb} and $F(u,v)=F(F(u,u),F(v,v))$ for all $u,v\in X$.
\end{enumerate}
\end{theorem}

For all integer $n\geq 1$, let $X_n=\{1,\ldots,n\}$ and let
$\alpha(n)$ (resp.\ $\beta(n)$) denote the number of
$(1,4)$-selective operations on $X_n$ that are idempotent (resp.\
the number of isomorphism types of $(1,4)$-selective operations on $X_n$ that are
idempotent). In the following propositions we show
that $\alpha(n)=A121860(n)$ and $\beta(n)=d(n)=A000005(n)$ (see
\cite{Oeis}), where $d(n)$ denotes the number of positive integer
divisors of $n\in \mathbb{N}$.

\begin{proposition}
For all integer $n\geq 1$, we have
$$\alpha(n)~=~\sum_{d|n}\frac{n!}{d!\left(\frac{n}{d}\right)!}$$
\end{proposition}

\begin{proof}
By Corollary \ref{cor:chiddb}, counting the number of
$(1,4)$-selective operations on $X_n$ that are idempotent is
equivalent to counting the number of ways to partition $X_n$ into
$k$ equivalence classes of sizes $l,\ldots,l$ and the number of
bijections between two consecutive equivalence classes. Thus, we
have
$$
\alpha(n)~=~ \sum_{k,l \atop kl = n }\frac{{n\choose
l,\ldots,l}}{k!}(l!)^{k-1} ~=~ \sum_{k,l \atop kl = n
}\frac{n!}{k!l!},
$$
where the multinomial coefficient ${n\choose l,\ldots,l}$ provides
the number of ways to put the elements $1,\ldots,n$ into $k$ classes
of sizes $l,\ldots,l$ and $l!$ is the number of bijections between
two such classes.
\end{proof}


\begin{proposition}
For all integer $n\geq 1$, we have $\beta(n)=d(n)$.
\end{proposition}

\begin{proof}
By Corollary \ref{cor:chiddb}, counting the number of isomorphism types of
$(1,4)$-selective operations on $X_n$ that are idempotent is equivalent to counting the number of ways to arrange
the elements of $X_n$ in an unlabeled grid form, where two elements
$x,y\in X_n$ belong to the same column (resp.\ row) if and only if
$x\sim_{F,1} y$ (resp.\ $x\sim_{F,2} y$). Thus, $\beta(n)$ provides
the number of ways to write $n$ into a product of two elements
$k,l\in \{1,\ldots,n\}$. This is in turn the number of divisors of
$n$.
\end{proof}

\begin{corollary}
$\alpha(n)=2$ (resp.\ $\beta(n)=2$) if and only if $n$ is prime.
\end{corollary}

\begin{corollary}\label{cor:pi}
Let $F\colon X_n^2 \to X_n$ be $(1,4)$-selective and idempotent. If
$n$ is prime, then $F=\pi_1$ or $F=\pi_2$.
\end{corollary}

\begin{remark}
From Corollary \ref{cor:pi}, it follows that the example of
$(1,4)$-selective and idempotent operation described in Remark
\ref{rem} is the smallest example that is neither $\pi_1$ nor
$\pi_2$.
\end{remark}

\section{$(1,2)$-Selectiveness}

In Lemma \ref{lem:s2} we already gave a characterization of
$(1,2)$-selective operations. As a corollary we get the following
result if $F$ is surjective.
\begin{corollary}\label{thmpr1}
Let $F:X^2\to X$ be an operation that is $(1,2)$-selective. Then the
following assertions are equivalent.
\begin{enumerate}
\item[(i)] $F=\pi_1$,
\item[(ii)] $F$ is quasitrivial,
\item[(iii)] $F$ is idempotent,
\item[(iv)] $F$ is surjective.
\end{enumerate}
\end{corollary}
\begin{proof}
$\textrm{(i)}\Rightarrow \textrm{(ii)}\Rightarrow
\textrm{(iii)}\Rightarrow \textrm{(iv)}$: Obvious.

$\textrm{(iv)}\Rightarrow \textrm{(i)}$: This follows from Lemma \ref{lem:s2}.  
\end{proof}





Now we characterize those operations that are bisymmetric.
\begin{proposition}\label{thmpr3}
A bisymmetric operation $F:X^2\to X$ is $(1,2)$-selective if and
only if the following two conditions hold.
\begin{enumerate}
\item[(i)] $F(x,y)=F(x,z)$ for all $x,y,z\in X$,
\item[(ii)] $F|_{\ran(F)\times X}=\pi_1|_{\ran(F)\times X}$.
\end{enumerate}
\end{proposition}

\begin{proof}
(Necessity) Let $x,y,z\in X$. Using bisymmetry and
$(1,2)$-selectiveness, we get
$$F(x,y)~=~F(F(x,y),F(z,z))~=~F(F(x,z),F(y,z))~=~F(x,z),$$
which proves (i).

Let $x\in \ran(F)$ and $y\in X$. Since $x\in \ran(F)$, there exist
$x_1, x_2\in X$ such that $F(x_1,x_2)=x$. Hence, using (i) and
$(1,2)$-selectiveness, we get
$$F(x,y)~=~F(x,x)~=~F(F(x_1, x_2), F(x_1, x_2))~=~F(x_1,x_2)~=~x,$$
which proves (ii).

(Sufficiency) Condition (ii) clearly implies that $F$ is
$(1,2)$-selective.

Let $x,y,u,v\in X$. Using condition (i) and $(1,2)$-selectiveness,
we get
$$
F(F(x,y),F(u,v)) ~=~ F(x,y) ~=~ F(x,u) ~=~ F(F(x,u),F(y,v)),
$$
which shows that $F$ is bisymmetric.
\end{proof}

In Figure \ref{fig:5} we illustrate a $(1,2)$-selective and
bisymmetric operation on $X$. The vertical lines express that
$F(x,\cdot)=x$ for all $x\in \ran(F)$. The dotted lines express that
the function $F$ is constant along those lines.

\begin{figure}[!ht]
    \centering

    \begin{tikzpicture}
    \draw (0,0) rectangle (4,4);
    \draw (0,0) rectangle (3,3);
    \draw (3,3)-- (3,4);
    \draw[dotted] (3.33,0)-- (3.33,4);
    \draw[dotted] (3.66,0)-- (3.66,4);

    \draw (0.33,0)-- (0.33,4);
    \draw (0.66,0)-- (0.66,4);
    \draw (1,0)-- (1,4);
    \draw (1.33,0)-- (1.33,1.7);
    \draw (1.66,0)-- (1.66,1.7);
    \draw (1.33,2.3)-- (1.33,4);
    \draw (1.66,2.3)-- (1.66,4);
    \draw (2,0)-- (2,4);
    \draw (2.33,0)-- (2.33,4);
    \draw (2.66,0)-- (2.66,4);
    \node at (3.7,-0.3) {$X\backslash\ran(F)$};
    \node at (-0.85,3.5) {$X\backslash\ran(F)$};
    \node at (1.5,2) {$\pi_1$};
    \node at (1.5,-0.3) {$\ran(F)$};
    \node at (-0.6,1.5) {$\ran(F)$};
    \end{tikzpicture}
     \caption{}
    \label{fig:5}
\end{figure}


In the following statement we provide a characterization of those
operations that are associative.
\begin{proposition}\label{prpr1}
An associative operation $F:X^2\to X$ is $(1,2)$-selective if and
only if the following two conditions hold.
\begin{enumerate}
\item[(i)] $F(x,y)=F(x,F(y,z))$ for all $x,y,z\in X$,
\item[(ii)] $F|_{\ran(F)\times X}=\pi_1|_{\ran(F)\times X}$.
\end{enumerate}
\end{proposition}

\begin{proof}
(Necessity) Let $x\in \ran(F)$ and $y\in X$. By Lemma \ref{lem:s2}
we have $F(x,x)=x$. Thus, using associativity and
$(1,2)$-selectiveness, we get
\begin{eqnarray*}
F(x,y) &=& F(F(x,x),y) \\
&=& F(F(F(x,x), x),y)~=~F(F(x,x), F(x,y))~=~F(x,x)~=~x,
\end{eqnarray*}
which proves (ii).

Let $x,y,z\in X$. Since $F(x,y)\in \ran(F)$, using (ii) and the
associativity of $F$, we get $F(x,y)=F(F(x,y),z)=F(x, F(y,z))$,
which proves (i).

(Sufficiency) Condition (ii) clearly implies that $F$ is
$(1,2)$-selective.

Let $x,y,z\in X$. Applying (ii) for $F(x,y)\in \ran(F)$ and (i) we
get $F(F(x,y),z)=F(x,y)=F(x,F(y,z))$, which shows that $F$ is
associative.
\end{proof}

As a consequence of Propositions \ref{thmpr3} and \ref{prpr1} we get
the following.

\begin{corollary}
Any $(1,2)$-selective and bisymmetric operation is associative.
\end{corollary}

Finally we show an example of $(1,2)$-selective and associative
operation that is not bisymmetric on $X=\{a,b,c,d\}$ (see Figure
\ref{fig:6}). The value $F(x,y)$ is represented above the
corresponding point $(x,y)$ in Figure \ref{fig:6} for all $x,y\in
X$. That is $F|_{\R^2_{F}}=\pi_1|_{\R^2_{F}}$ and
$F(d,a)=F(d,b)=F(d,d)=a$ and $F(d,c)=b$. By Lemma \ref{lem:s2}, $F$
is $(1,2)$-selective. It is also clear that $F$ is not bisymmetric
by Proposition \ref{thmpr3}. Using Proposition \ref{thmpr3} it can
be easily shown that $F$ is associative.

\begin{figure}[!ht]
    \centering

    \begin{tikzpicture}
    \draw[fill=black] (0,0) circle (0.05);
    \draw[fill=black] (1,0) circle (0.05);
    \draw[fill=black] (2,0) circle (0.05);
    \draw[fill=black] (3,0) circle (0.05);
    \draw[fill=black] (0,1) circle (0.05);
    \draw[fill=black] (1,1) circle (0.05);
    \draw[fill=black] (2,1) circle (0.05);
    \draw[fill=black] (3,1) circle (0.05);
    \draw[fill=black] (0,2) circle (0.05);
    \draw[fill=black] (1,2) circle (0.05);
    \draw[fill=black] (2,2) circle (0.05);
    \draw[fill=black] (3,2) circle (0.05);
      \draw[fill=black] (0,3) circle (0.05);
    \draw[fill=black] (1,3) circle (0.05);
    \draw[fill=black] (2,3) circle (0.05);
    \draw[fill=black] (3,3) circle (0.05);
       \draw[->] (-1,-0.5) -- (4,-0.5);
       \draw[->] (-0.5,-1) -- (-0.5,4);
    \draw (0,-0.6)--(0, -0.4);
    \draw (1,-0.6)--(1, -0.4);
    \draw (2,-0.6)--(2, -0.4);
    \draw (3,-0.6)--(3, -0.4);
    \draw (-0.6,0)--(-0.4,0);
    \draw (-0.6,1)--(-0.4,1);
    \draw (-0.6,2)--(-0.4,2);
    \draw (-0.6,3)--(-0.4,3);

    \node at (0,-0.9) {$a$};
    \node at (1,-0.9) {$b$};
    \node at (2,-0.9) {$c$};
       \node at (3,-0.9) {$d$};
    \node at (-0.9,0) {$a$};
    \node at (-0.9,1) {$b$};
    \node at (-0.9,2) {$c$};
    \node at (-0.9,3) {$d$};

    \node at (0,0.3 ) {$a$};
    \node at (0,1.3 ) {$a$};
    \node at (0,2.3 ) {$a$};
    \node at (0,3.3 ) {$a$};
     \node at (1,0.3 ) {$b$};
    \node at (1,1.3 ) {$b$};
    \node at (1,2.3 ) {$b$};
    \node at (1,3.3 ) {$b$};
     \node at (2,0.3 ) {$c$};
    \node at (2,1.3 ) {$c$};
    \node at (2,2.3 ) {$c$};
    \node at (2,3.3 ) {$c$};
     \node at (3,0.3 ) {$a$};
    \node at (3,1.3 ) {$a$};
    \node at (3,2.3 ) {$b$};
    \node at (3,3.3 ) {$a$};
    \end{tikzpicture}
     \caption{}
    \label{fig:6}
\end{figure}
\section{Conclusion and further directions}
In this article we introduced and investigated the $(i,j)$-selective
operations. First we showed some basic properties of these
operations. As a consequence we proved that any 
$(i,j)$-selective operation with $j<i$ and any $(2,3)$-selective
operation is constant. We also characterized $(i,i)$-selective
operations. We described $(1,3)$-selective operations using the
equivalence relation $\sim_{F}$ and it turned out that it is enough
to understand $\id(F)$ (the set of idempotent elements). We also
proved that $(1,4)$-selective operations are bisymmetric and
associative. We characterized $(1,4)$-selective operations using
equivalence relations $\sim_{F,1}$ and $\sim_{F,2}$. Finally we
described $(1,2)$-selective operations. We studied the relation of
these operations with associativity, bisymmetry and other basic
properties.

In view of these results some questions arise. Now, we list them
below.

\begin{itemize}
    \item
    Let $n\geq 3$ be an integer and let $i_1,\ldots,i_n\in \{1,\ldots,n^2\}$.
    We say that an operation $F \colon X^n\to X$ is {\it $(i_1,\ldots,i_n)$-selective}, if
    $$F(F(x_{1},\ldots, x_{n}), \ldots, F(x_{n(n-1)+1}, \ldots,x_{n^2}))~=~F(x_{i_1}, \ldots, x_{i_n}),$$
    for all $x_{1},\ldots,x_{n^2}\in X$.\\
    Find characterizations of the class of $(i_1,\ldots,i_n)$-selective operations.
    \item Let $i,j \in \{1,2,3,4\}$. We say that the operations $F,G,H,K\colon X^2 \to X$ are {\it generalized $(i,j)$-selective} if
    $$F(G(x_1,x_2), H(x_3,x_4))~=~K(x_i, x_j), \qquad x_1,x_2,x_3,x_4\in X.$$
    Find characterizations of the class of generalized $(i,j)$-selective operations.
    \item Recall that an operation $F\colon X^2 \to X$ is said to be \emph{permutable} \cite{Acz2006,Bus2014} if it satisfies the following functional equation
    $$
    F(F(x,y),z) ~=~ F(F(x,z),y), \qquad x,y,z \in X.
    $$
    We observe that any $(1,2)$-selective operation that is bisymmetric is permutable. Find the conditions under which an $(i,j)$-selective operation is permutable.
\end{itemize}

\section*{Acknowledgements}

The authors would like to thank Jean-Luc Marichal for fruitful
discussions and valuable remarks. This research is supported by the
internal research project R-AGR-0500 of the University of Luxembourg
and by the Luxembourg National Research Fund R-AGR-3080.

\end{document}